\documentclass[11pt,letterpaper]{amsart}

\usepackage{amssymb,amsmath,amsthm,amstext,amscd,latexsym,graphics,graphicx,bbm,caption, mathtools}
\usepackage{relsize}
\usepackage[usenames,dvipsnames,svgnames,table]{xcolor}
\usepackage{hyperref}
\usepackage{tikz}
\usetikzlibrary{positioning}
\usepackage{tikz-cd}
\tikzset{>=stealth}
\hypersetup{plainpages=false,colorlinks=true, pagebackref}
\hypersetup{citecolor=Sepia,linkcolor=blue, urlcolor=blue}

\newtheorem{intro-thm}{Theorem}[]
\theoremstyle{plain}
\newtheorem{thm}{Theorem}[section]
\newtheorem{theorem}[thm]{Theorem}

\newtheorem{lemma}[thm]{Lemma}
\newtheorem{corollary}[thm]{Corollary}
\newtheorem{proposition}[thm]{Proposition}

\theoremstyle{definition}
\newtheorem{remark}[thm]{Remark}

\newtheorem{definition}[thm]{Definition}
\newtheorem{example}[thm]{Example}

\newcommand{\Spec}{{\rm Spec \,}}

\renewcommand{\tilde}{\widetilde}
\input{xy}
\xyoption{all}

\title[$\mathbb{A}^1$-homotopy type of $\mathbb{A}^2 \setminus \left\{(0,0) \right\}$]{$\mathbb{A}^1$-homotopy type of $\mathbb{A}^2 \setminus \left\{(0,0) \right\}$}
\author{Utsav Choudhury}
\address{Indian Statistical Institute\\ Stat Math Unit\\203 B.T. Road \\ Kolkata - 700108\\ India}
\email{utsav@isical.ac.in}
\author{Biman Roy}

\address{Indian Statistical Institute\\ Stat Math Unit\\203 B.T. Road \\ Kolkata - 700108\\ India}
\email{bimanroy31@gmail.com}
\keywords{$\mathbb{A}^1$-homotopy theory, Affine Algebraic Geometry, Zariski Cancellation}
\subjclass[2010]{Primary 14F42}
\begin{document}
\maketitle
\begin{abstract}
In this article we prove that any  $\mathbb{A}^1$-connected smooth $k$-variety is $\mathbb{A}^1$-uniruled for any algebraically closed field $k$. We establish that if a non empty open subscheme $X$ of a smooth affine $k$-scheme is $\mathbb{A}^1$-weakly equivalent to $\mathbb{A}^2_{k} \setminus \left\{(0,0) \right\}$,  then $X \cong \mathbb{A}^2_{k} \setminus \left\{(0,0) \right\}$ as $k$-varieties for any field $k$ of characteristic $0$.
\end{abstract}
\tableofcontents
%\begin{theorem}
\section{Introduction}
Let $k$ be a field and $Sm/k$ be the category of smooth separated finite type $k$-schemes. A scheme $X \in Sm/k$ is said to be \textbf{$\mathbb{A}^1$-uniruled or log-uniruled} if there is a dominant generically finite morphism $H: \mathbb{A}^1_k \times_k Y \to X$ for some $k$-variety $Y$. 

  The unstable motivic homotopy category constructed by Morel-Voevodsky \cite{mv}, denoted by $\mathbf{H}(k)$, is obtained by attaching simplicial homotopy to $Sm/k$ and then inverting the Nisnevich local equivalences and the projection maps $pr : X \times \mathbb{A}^1 \to X$ for any $X \in Sm/k$. Objects in $\mathbf{H}(k)$ are the simplicial presheaves on $Sm/k$ and any such object $\mathcal{X}$ will be called a space. For example, any $X \in Sm/k$ gives an object in $\mathbf{H}(k)$. Recall that a morphism of spaces $f : \mathcal{X} \to \mathcal{Y}$ is called an $\mathbb{A}^1$-weak equivalence if the class of $f$ in $\mathbf{H}(k)$ is an isomorphism. There is a pointed version of this construction, denoted by $\mathbf{H}_{\bullet}(k)$, is called the pointed unstable $\mathbb{A}^1$-homotopy category. This model category is constructed from the category of pointed spaces and the morphism of pointed spaces and inverting $\mathbb{A}^1$-weak equivalences of pointed spaces. A morphism of pointed spaces $f : (\mathcal{X}, x) \to (\mathcal{Y}, y)$ is an $\mathbb{A}^1$-weak equivalence if the corresponding morphism $f : \mathcal{X} \to \mathcal{Y}$, after forgetting the base point, is an isomorphism in $\mathbf{H}(k)$.  For any space $\mathcal{X}$, define $\pi_0^{\mathbb{A}^1}(\mathcal{X})$ to be the Nisnevich sheaf associated to the presheaf 
$$U \in Sm/k \mapsto Hom_{\mathbf{H}(k)}(U, \mathcal{X}).$$
This is called the $\mathbb{A}^1$-connected component sheaf of the space $\mathcal{X}$.  Classification of $\mathbb{A}^2_k$ using $\mathbb{A}^1$-connected component sheaves was described in \cite[Theorem 4.9, Theorem 5.1]{cb}. Note that in the proof of \cite[Theorem 4.9]{cb} we only need that $\pi_0^{\mathbb{A}^1}(X)(Spec \ k)$ is trivial, for a smooth surface $X \in Sm/k$. Our first main result of this article is Theorem \ref{A1-connected dominant}, where we show that any $\mathbb{A}^1$-connected variety $X$ is $\mathbb{A}^1$-uniruled. This can be seen as a generalisation of \cite[Corollary 2.4]{bhs1} to the non proper case.
\

Next we try to analyse the case of $\pi_0^{\mathbb{A}^1}(X)(Spec \ k) = \bullet$, for any smooth affine surface $X \in Sm/k$, where $k$ is an algebraically closed field. Proposition \ref{connected implies} shows that  in this case either $X$ is $\mathbb{A}^1$-uniruled or through any point of $X$ there is an $\mathbb{A}^1$ and moreover there exists a point through which two distinct $\mathbb{A}^1$ pass. In the proof of Theorem \ref{A1-connected dominant} and Proposition \ref{connected implies} we exploit the beautiful idea of $\mathbb{A}^1$-ghost homotopies developed in \cite{bhs}.  For other applications of $\mathbb{A}^1$-ghost homotopy, see \cite{bs}, \cite{brs}, \cite{bhs2}.

\

 In \cite[Theorem 1.1]{cb} we showed that for any smooth affine surface $X$ over $k$, where $k$ is a field of characteristic $0$, if the structure map $f : X \to Spec \ k$ is an $\mathbb{A}^1$-weak equivalence, then $X \cong \mathbb{A}^2_k$ as $k$-varieties. There are two circles $S^1_s, S^1_t$, the simplicial and the Tate circles respectively, in $\mathbf{H}(k)$ (\cite[Section 3.2]{mv}). These give us the following spheres for $i \geq 0$, $$S^i_s := (S^1_s)^{\wedge i}, \ S^i_t := (S^1_t)^{\wedge i}$$ and for $i \geq j \geq 0$ we have the following mixed spheres $$S^{i,j} := S_s^{i-j} \wedge S^j_t.$$  The quasi-affine varieties $\mathbb{A}^n_k \setminus \left\{(0,\dots, 0) \right\}$ are also mixed spheres in $\mathbf{H}_{\bullet}(k)$, infact $\mathbb{A}^n_k \setminus \left\{(0,\dots, 0) \right\} \cong S^{2n-1, n}$ in $\mathbf{H}_{\bullet}(k)$ (\cite[Section 3.2, Example 2.20]{mv}). Therefore, it is natural to ask whether a pointed open subscheme $X$ of a smooth affine $k$-variety of dimension $n$, $\mathbb{A}^1$-weakly equivalent to $S^{2n-1, n}$ in $\mathbf{H}_{\bullet}(k)$ is isomorphic as $k$-variety to $\mathbb{A}^n_k \setminus \left\{(0,\dots, 0) \right\}$. We prove that for $n =2$, $\mathbb{A}^2_{k} \setminus \left\{(0, 0) \right\}$ is the only pointed open subvariety of a smooth affine variety (upto isomorphism) of dimension $2$ over a field of characteristic zero, which is $\mathbb{A}^1$-weakly equivalent to $S^{3,2}$ in $\mathbf{H}_{\bullet}(k)$ (Theorem \ref{poincare}). We also show that there are smooth quasi-affine $k$-threefolds which are not isomorphic to $\mathbb{A}^3_{k} \setminus \left\{(0,0, 0) \right\}$ but $\mathbb{A}^1$-homotopic in $\mathbf{H}_{\bullet}(k)$ to $S^{5,3}$ (Theorem \ref{not poincare}). 

\

\noindent\textbf{Acknowledgements:} The authors would like to thank Chetan Balwe, Fabien Morel, Aravind Asok, Anand Sawant, Joseph Ayoub, Suraj Yadav, Anand Sawant for their helpful comments and suggestions. We want to speacially thank Chetan Balwe for giving suggestions to make the proof of Proposition \ref{induction separable} simpler.
%We thank the anonymous referee for insightful suggestions which helped us to improve the exposition. 
This article is a part of the Ph.D thesis of the second author and the second author would like to thank Indian Statistical Institute for providing all the resources during this work.

\section{$\mathbb{A}^1$-connectedness and log uniruledness}
In this section we show that an $\mathbb{A}^1$-connected variety contains a family of affine lines (Theorem \ref{A1-connected dominant}). We also prove that a smooth affine $k$-surface $X$ with $\pi_0^{\mathbb{A}^1}(X)(Spec \ k)$ trivial, either contains a dominant family of affine lines in $X$ or through every point there is an affine line in $X$ and there exists a point $x \in X(k)$ through which intersecting $\mathbb{A}^1$'s pass  (Proposition \ref{connected implies}). Throughout the section we assume $k$ to be an algebraically closed field unless otherwise mentioned.

\

%The next proposition (Proposition \ref{induction separable}) gives a way of constructing non-constant $\mathbb{A}^1$-homotopy in $\mathcal{S}^m(X)$ starting from a non-constant homotopy in $\mathcal{S}^n(X)$ whenever $n > m$. In this respect the proposition is similar to \cite[Proposition 4.7]{cb}. However, the difference is that in \cite[Proposition 4.7]{cb} the algorithm requires fixing a closed point but in the following proposition the argument is more generic in nature. \\
Suppose $\mathcal{F}$ is a Nisnevich sheaf on $Sm/k$ and $X, W \in Sm/k$. A non-constant homotopy (\cite[Definition 4.4]{cb}) $H \in \mathcal{F}(\mathbb{A}^1_W)$ is such that $H(0) \neq H(1) \in \mathcal{F}(W)$, where $H(0)$ and $H(1)$ are induced by the $0$-section and the $1$-section from $W$ to $\mathbb{A}^1_W$ respectively.
%\begin{definition}
%A homotopy $H \in \mathcal{F}(\mathbb{A}^1_W)$ is said to be a non-constant homotopy if $H(0) \neq H(1) \in \mathcal{F}(W)$, where $H(0)$ and $H(1)$ are induced by the $0$-section and the $1$-section from $W$ to $\mathbb{A}^1_W$ respectively.
%\end{definition}
\begin{remark} \label{useful remark}
 %Let $\mathcal{F}$ be a sheaf and $X \in Sm/k$. 
%A ghost homotopy arises when we lift a section $H \in \mathcal{S}(\mathcal{F})()$
A ghost homotopy \cite[Definition 3.2]{bhs} $H \in \mathcal{S}(\mathcal{F})(\mathbb{A}^1_W)$ is given by an element $\phi \in \mathcal{F}(V)$, where $V \to \mathbb{A}^1_W$ is a Nisnevich covering such that there is a Nisnevich covering $V^{\prime} \to V \times_{\mathbb{A}^1_W} V$ along with there is a chain of $\mathbb{A}^1$-homotopies $G_1, G_2,.., G_n \in \mathcal{F}(\mathbb{A}^1_{V^{\prime}})$ such that $G_1(0) = p_1^*(\phi)|_{V^{\prime}}$ and $G_n(1)=p_2^*({\phi})|_{V^{\prime}}$ (where $p_1, p_2 : V \times_{\mathbb{A}^1_W} V \to V$ are the projection maps). If $ p_1^*(\phi)|_{V^{\prime}} =  p_2^*(\phi)|_{V^{\prime}}$, then $H$ can be lifted to an element $H^{\prime} \in \mathcal{F}(\mathbb{A}^1_W)$. Therefore if $H$ is a non-constant homotopy, then $ p_1^*(\phi)|_{V^{\prime}} \neq  p_2^*(\phi)|_{V^{\prime}}$ and thus we can assume all the $\mathbb{A}^1$-homotopies $G_i$'s are non-constant \cite[Remark 4.5]{cb}.

%A section $\alpha \in S(\mathcal{F})(X)$ is given by a  Nisnevich covering $W \to X$, a section $\gamma \in \mathcal{F}(W)$ and a Nisnevich covering $W' \to W \times_{X} W$ such that $ p_1^*(\gamma)|_{W'} $ and $p_2^*(\gamma)|_{W'}$  in $\mathcal{F}(W')$ are naively $\mathbb{A}^1$-homotopic. 
%joined by a chain of $\mathbb{A}^1$-homotopies 
%(where $p_1, p_2 : W \times_{X} W \to W$ are the projection maps). If  $ p_1^*(\gamma)|_{W'} = p_2^*(\gamma)|_{W'}$, then $\gamma$ can be lifted to some element $\alpha' \in \mathcal{F}(X)$ and in this case $\alpha'$ maps to $\alpha$ via the canonical morphism $\mathcal{F} \to S(\mathcal{F})$.  Otherwise, we will get an element $H \in  \mathcal{F}(\mathbb{A}^1_{W'})$ such that $ p_1^*(\gamma)|_{W'} = H(0) \neq H(1)$ as sections in $\mathcal{F}(W^{\prime})$. This is essentially the data of ghost homotopy mentioned in \cite[Definition 3.2]{bhs}. 
\end{remark}
The next proposition (Proposition \ref{induction separable}) gives a way of constructing non-constant $\mathbb{A}^1$-homotopy in $\mathcal{S}^m(X)$ starting from a non-constant homotopy in $\mathcal{S}^n(X)$ whenever $n > m$. In this respect the proposition is similar to \cite[Proposition 4.7]{cb}. However, the difference is that in \cite[Proposition 4.7]{cb} the algorithm requires fixing a closed point but in the following proposition the argument is more generic in nature. 
\begin{proposition} \label{induction separable}
Suppose $X, W \in Sm/k$ are irreducible schemes, $n \geq 1$ and $H \in \mathcal{S}^n(X)(\mathbb{A}^1_W)$ is a non-constant homotopy such that $H(0)$ factors through $X$ and $H(0) : W \to X$ is a dominant morphism. Then there is some $m < n$, $W^{\prime} \in Sm/k$ irreducible and a non-constant homotopy $H^{\prime} \in \mathcal{S}^m(X)(\mathbb{A}^1_{W^{\prime}})$ such that $H^{\prime}(0)$ factors through $X$ and $H^{\prime}(0):W^{\prime} \to X$ is a dominant morphism.
\end{proposition}
\begin{proof}

The first part of the proof goes as in \cite[Proposition 4.7]{cb}. The canonical morphism $\eta: X \to \mathcal{S}^n(X)$ is an epimorphism. So given $H \in \mathcal{S}^n(X)(\mathbb{A}^1_W)$, there is a Nisnevich covering $f:V \to \mathbb{A}^1_W$ along with a morphism $\phi: V \to X$ making the diagram
 $$\begin{tikzcd}
 V \ar[r, "\phi"] \ar[d, "f"] &X \ar[d]\\
\mathbb{A}^1_{W} \ar[r,"H"] 
&\mathcal{S}^n(X)
\end{tikzcd}$$ 
commutative and there is a Nisnevich covering $V^{\prime} \to V\times_{\mathbb{A}^1_W} V$ along with a chain of non-constant (since $H$ is a non-constant homotopy, so $p_1^{*}(\phi)|_{V^{\prime}} \neq p_2^*(\phi)|_{V^{\prime}} \in \mathcal{S}^{n-1}(X)(V^{\prime})$ by Remark \ref{useful remark}, here $p_1, p_2: V \times_{\mathbb{A}^1_W} V \to V$ are the projection maps) $\mathbb{A}^1$-homotopies $G_1, G_2, \dots, G_t \in \mathcal{S}^{n-1}(X)(\mathbb{A}^1_{V^{\prime}})$ such that 
$$G_1(0) = p_1^*(\phi)|_{V^{\prime}} \text{ and } G_t(1) = p_2^*(\phi)|_{V^{\prime}}.$$ \par
  Suppose that $V_i$'s, $1 \leq i \leq q$ are the irreducible components of $V$, which are also the connected components. Then each irreducible component $V_0^{\prime}$ of $V^{\prime}$ maps to $V_i \times_{\mathbb{A}^1_W} V_j$ for some $i, j$ (note that each $V_i \times_{\mathbb{A}^1_W} V_j$ is non-empty, since $W$ is irreducible) and there are dominant maps (\'etale maps) from $V_0^{\prime}$ to $V_i$ (for some $i$) induced by $p_1$ and $p_2$. Consider the following commutative diagram induced by the $0$-section of $H$:
$$\begin{tikzcd}[column sep=50pt, row sep=50 pt]
U \ar[r, "\theta"] \ar[d, "f^{\prime}"] &V \ar[r, "\phi"] \ar[d, "f" near start] &X \ar[d]\\
W \ar[r,"i_0" below] \ar[urr, dotted, "H(0)" near start] &\mathbb{A}^1_{W} \ar[r, "H" below] &\mathcal{S}^n(X)
\end{tikzcd}$$
where $i_0 : W \to \mathbb{A}^1_W$ is the $0$-section and the left square is cartesian. Thus $f^{\prime}$ is a Nisnevich covering. Here the lower triangle is commutative, since $H(0)$ factors through $X$. The upper triangle is not commutative in general. However after shrinking $W$, we can always assume that $H(0)$ lifts to $V$ i.e. there is a morphism $\theta^{\prime}: W \to V$ such that $H(0)$ factors as the morphism $\theta^{\prime}: W \to V$, followed the morphism $\phi: V\to X$. Indeed, suppose that $U_1, U_2,.., U_d$ are the connected components of $U$. If all the homotopies $H_i:\mathbb{A}^1_{U_i} \to \mathcal{S}^n(X)$ induced by the composition 
$$\mathbb{A}^1_{U_i} \xrightarrow{(Id, f^{\prime})} \mathbb{A}^1_W \xrightarrow{H} \mathcal{S}^n(X)$$
are constant i.e. $H_i(0) = H_i(1)$, then $H(0) = H(1)$. It is not possible, since $H$ is a non-constant homotopy. Therefore there is some $i$ such that $H_i$ is non-constant. The $0$-section $H_i(0) = H(0)|_{U_i}$, so $H_i(0)$ factors through $X$ and $H_i(0): U_i \to X$ is dominant. We then replace the homotopy $H$ by $H_i$. Therefore we always have the following commutative diagram (after shrinking $W$):
$$\begin{tikzcd}[column sep=50pt, row sep=50 pt]
 &V \ar[r, "\phi"] \ar[d, "f" near start] &X \ar[d]\\
W \ar[r, "i_0" below]  \ar[ru, "\theta"] \ar[urr, dotted] &\mathbb{A}^1_{W} \ar[r, "H" below] &\mathcal{S}^n(X)
\end{tikzcd}
$$
Since $W$ is irreducible and $H(0)$ is dominant, so $W$ maps to some irreducible component of $V$ (say $V_l$) and $\phi|_{V_l}$ is dominant.

We analyse the following cases:

%We have the following cases

%\textbf{Case 1:} Suppose $\phi \circ \theta \neq H(0) \circ f^{\prime}$ as morphisms to $X$. Consider the following diagram:  
%$$\begin{tikzcd}[column sep=50pt, row sep=50 pt]
%W^{\prime} \ar[r, "\theta"] \ar[d, "f^{\prime}"] &V \ar[r, "\phi"] \ar[d, "f" near start] &X \ar[d]\\
%W \ar[r,"i_0" below] \ar[urr, dotted, "H(0)" near start] &\mathbb{A}^1_{W} \ar[r, "H" below] &\mathcal{S}^n(X)
%\end{tikzcd}$$
%where $i_0 : W \to \mathbb{A}^1_W$ is the $0$-section. Here the left square is cartesian and the lower triangle is commutative, since $H(0)$ factors through $X$. The upper triangle is not commutative in this case, however $\phi \circ \theta$ and $H(0) \circ f^{\prime}$

%\textbf{Case 1:} Suppose $\phi \circ \theta \neq H(0) \circ f^{\prime}$ as morphisms to $X$. The upper triangle is not commutative in this case, however $\phi \circ \theta$ and $H(0) \circ f^{\prime}$ are same in $\mathcal{S}^n(X)(W^{\prime})$. Suppose $m \geq 0$ is the least such that these two maps are same in $\mathcal{S}^{m+1}(X)(W^{\prime})$. Thus there is a Nisnevich covering $W^{\prime \prime} \to W^{\prime}$ and there is a non-constant homotopy $H^{\prime} \in \mathcal{S}^{m}(X)(\mathbb{A}^1_{W^{\prime \prime}})$, such that $H^{\prime}(0) = H(0)|_{W^{\prime \prime}}$. There is an irreducible component (say $W_0$) of $W^{\prime \prime}$ such that $H^{\prime}|_{\mathbb{A}^1_{W_0}}$ is non-constant. Since the map $W_0 \to W$ is \`etale and $H(0)$ is dominant morphism, $H^{\prime}|_{\mathbb{A}^1_{W_0}}(0)$ is dominant morphism.  \par
\textbf{Case 1:} Suppose that $\overline{\phi(V_i)}$ are same for all $i$. 
%Since $f^{\prime}$ is a Nisnevich covering and $H(0)$ is dominant, $\phi$ is dominant. As $X$ is irreducible, 
In this case $\phi|_{V_i}$ is also dominant for every $i$. 
%The map $\phi|_{V_i}$ is also separable by Lemma \ref{separable} considering the upper triangle at the level of tangent space. 
There is some irreducible component $V_0^{\prime}$ of $V^{\prime}$ such that $G_1|_{\mathbb{A}^1_{V_0^{\prime}}}$ is a non-constant homotopy. If $V_0^{\prime}$ maps to $V_j \times_{\mathbb{A}^1_{W}} V_l$ for some $j, l$, the composition $$V_0^{\prime} \to V_j \times_{\mathbb{A}^1_{W}} V_l \xrightarrow{p_1} V_j \xrightarrow{\phi|_{V_j}} X$$
is dominant since the map $V_0^{\prime} \to V_j$ is \'etale. Thus in this case $G_1|_{\mathbb{A}^1_{V_0^{\prime}}}$ is the required homotopy. \par
 % \item
\textbf{Case 2:} Suppose that there are $i, j$ such that  $\overline{\phi(V_i)} \neq \overline{\phi(V_j)}$. 
%Since $f^{\prime}$ is a Nisnevich covering and $H(0)$ is dominant, $\phi$ is dominant. As $X$ is irreducible, there is some $s$ such that $\overline{\phi(V_s)} = X$. 
Thus if needed renumbering all $V_m$-s we can assume that $\overline{\phi(V_1)}, \overline{\phi(V_2)}, \dots, \overline{\phi(V_s)} = X$ and $ \overline{\phi(V_{s+1})}, \dots, \overline{\phi(V_q)} \subsetneqq X$. If there is some irreducible component $V_0^{\prime}$ of $V^{\prime}$ that maps to $V_m \times_{\mathbb{A}^1_W} V_l$ with $m \leq s$ and $p_1^*(\phi)|_{V_0^\prime} \neq p_2^*(\phi)|_{V_0^\prime}$ in $\mathcal{S}^{n-1}(X)(V_0^\prime)$, then same as Case 2 there is some $p$ such that $G_p|_{\mathbb{A}^1_{V_0^\prime}}$ is a non-constant homotopy and $G_p|_{\mathbb{A}^1_{V_0^\prime}}(0) = p_1^*(\phi)|_{V_0^\prime}$. 

%Considering the upper triangle for each $m \leq k$ at the level of tangent spaces, we have $\phi|_{V_m}$ is separable for each $m \leq k$. 
Thus we can assume that for each irreducible component $V_0^{\prime}$ of $V^{\prime}$ that maps to $V_m \times_{\mathbb{A}^1_W} V_l$ with $m \leq s$ we have, $p_1^*(\phi)|_{V_0^{\prime}} = p_2^*(\phi)|_{V_0^\prime}$ in $\mathcal{S}^{n-1}(X)(V_0^\prime)$. 
%Otherwise the conclusion follows from Case 2. 
Therefore we have for every $m \leq s$, 
\begin{gather*}
p_1^*(\phi)|_{V_m \times_{\mathbb{A}^1_W} V_l} = p_2^*(\phi)|_{V_m \times_{\mathbb{A}^1_W} V_l} \in \mathcal{S}^{n-1}(X)(V_m \times_{\mathbb{A}^1_W} V_l).
\end{gather*}
   Here we have following two subcases: \par
                 \textbf{Subcase 1:}  Suppose that there is a $t < n-1$ and an irreducible component $W_0^\prime$ of $V^{\prime}$ that maps to $V_m \times_{\mathbb{A}^1_W} V_l$ for some $m,l$ with $m \leq s$ such that $$p_1^*(\phi)|_{W_0^\prime} \neq p_2^*(\phi)|_{W_0^\prime} \in \mathcal{S}^{t}(X)(W_0^\prime).$$ Since $p_1^*(\phi)|_{W_0^\prime}$ and $p_2^*(\phi)|_{W_0^\prime}$ are same in $\mathcal{S}^{n-1}(X)(W_0^\prime)$, we choose $t$ such that $p_1^*(\phi)|_{W_0^\prime}$ is same with $p_2^*(\phi)|_{W_0^\prime}$ in $\mathcal{S}^{t+1}(X)(W_0^\prime)$ and $$p_1^*(\phi)|_{W_0^\prime} \neq p_2^*(\phi)|_{W_0^\prime} \in \mathcal{S}^{t}(X)(W_0^\prime).$$ 
Then there is a Nisnevich covering $V^{\prime \prime} \to W_0^\prime$ and a non-constant homotopy (by Remark \ref{useful remark}) $H^{\prime} \in \mathcal{S}^t(X)(\mathbb{A}^1_{V^{\prime \prime}})$ such that $H^{\prime}(0) = p_1^*(\phi)|_{V^{\prime \prime}}$. So there is an irreducible component $W^{\prime\prime}_0$ of $V^{\prime \prime}$ such that $H^{\prime}|_{\mathbb{A}^1_{W^{\prime\prime}_0}}$ is non-constant. Since the map $W_0^{\prime\prime} \to V_m$ is \'etale and $\phi|_{V_m}$ is dominant, $H^{\prime}|_{\mathbb{A}^1_{W^{\prime\prime}_0}}(0)$ is dominant.  \par
                  \textbf{Subcase 2:} Suppose that for every irreducible component $V_0^\prime$ of $V^{\prime}$ that maps to $V_m \times_{\mathbb{A}^1_W} V_l$ for some $m \leq s$, we have $p_1^*(\phi)|_{V_0^\prime} = p_2^*(\phi)|_{V_0^\prime}$ as morphisms to $X$. Therefore we have, 
\begin{gather*} 
p_1^*(\phi)|_{V_m \times_{\mathbb{A}^1_W} V_l} = p_2^*(\phi)|_{V_m \times_{\mathbb{A}^1_W} V_l}, \ \forall m \leq s \ \forall l
\end{gather*}
as morphisms to $X$. But then all $\overline{\phi(V_l)}$ are same for every $l$, since $p_1: V_m \times_{\mathbb{A}^1_W} V_l \to V_m$ and $p_2: V_m \times_{\mathbb{A}^1_W} V_l \to V_l$ are dominant (\'etale) maps. It is a contradiction, since we have assumed there are $i, j$ such that  $\overline{\phi(V_i)} \neq \overline{\phi(V_j)}$.  \par
 %\end{enumerate}
Therefore, the proposition is proved.
\end{proof}
\begin{theorem} \label{A1-connected dominant}
Suppose $X \in Sm/k$ is an irreducible, $\mathbb{A}^1$-connected scheme. Then there is some $W \in Sm/k$ with $W$  irreducible and a non-constant homotopy $H: \mathbb{A}^1_W \to X$ such that $H$ is a dominant morphism.
\end{theorem}
\begin{proof}

As $X$ is $\mathbb{A}^1$-connected, the sheaf $\mathcal{L}(X)$ is trivial \cite[Corollary 2.18]{bhs} and $X$ has a $k$-rational point (say $x_0 \in X(k)$). Consider two morphisms $Id: X \to X$ (the identity map on $X$) and $C_{x_0}:X \to Spec \ k \xrightarrow{x_0} X$. Since $\mathcal{L}(X)$ is trivial, $Id$ and $C_{x_0}$ are same in $\mathcal{L}(X)(X)$. 
%Then there is a non-constant homotopy $G \in \mathcal{S}^n(X)(\mathbb{A}^1_W)$ for some $n$, $W \in Sm/k$ irreducible such that $$n $n \geq 0$ and a non-empty open set $U$ of $X$ such that the inclusion $i : U \hookrightarrow X$ is same with $C_{x_0}|_{U}$ in $\mathcal{S}^{n+1}(X)(U)$. 
Then there is a Nisnevich covering $f : Y \to X$ and some $m \geq 1$ such that $f$ and $C_{x_0}|_Y$ are same in $\mathcal{S}^{m}(X)(Y)$. Choose least $n \geq 0$ such that $f$ and $C_{x_0}|_Y$ are same in $\mathcal{S}^{n+1}(X)(Y)$and they are not same in $\mathcal{S}^n(X)(Y)$. Thus there is a Nisnevich covering $f^{\prime}:Y^{\prime} \to Y$ and there are chain of non-constant $\mathbb{A}^1$-homotopies $G_1, G_2,..., G_t \in \mathcal{S}^n(X)(\mathbb{A}^1_{Y^{\prime}})$ such that $G_1(0) = f \circ f^{\prime}$ and $G_t(1)=C_{x_0}|_{Y^{\prime}}$. There is some irreducible component $Y_0$ of $Y^{\prime}$ such that $G_1|_{\mathbb{A}^1_{Y_0}}$ is non-constant. Since the restriction $f \circ f^{\prime}|_{Y_0}: Y_0 \to X$ is \'etale, $G_1|_{\mathbb{A}^1_{Y_0}}(0)$ is dominant. 
%Thus there is an irreducible component $Y_0$ of $Y$ such that $f|_{Y_0} \neq C_{x_0}|_{Y_0}$ in $\mathcal{S}^n(X)(Y_0)$. The set $U = f(Y_0)$ is an open subset of $X$ and $f|_{Y_0} : Y_0 \to U$ is a Nisnevich covering, so we have $i = C_{x_0}|_{U}$ in $\mathcal{S}^{n+1}(X)(U)$ but $i \neq C_{x_0}|_{U}$ in $\mathcal{S}^n(X)(U)$. There is a Nisnevich covering $U^{\prime} \to U$ and a non-constant $\mathbb{A}^1$-homotopy $G \in \mathcal{S}^n(X)(\mathbb{A}^1_{U^{\prime}})$ such that $G(0) = i|_{U^{\prime}}$. There is some irreducible component $U_0$ of $U^{\prime}$ such that $G|_{\mathbb{A}^1_{U_0}}$ is non-constant. Here $G(0) = i|_{U_0}$ is separable morphism. 
Therefore applying Proposition \ref{induction separable} repeatedly, there is a non-constant homotopy $H : \mathbb{A}^1_{W} \to X$ for some $W \in Sm/k$ irreducible such that $H(0)$ is a dominant morphism. 
%Therefore $H$ is separable by Lemma \ref{separable section}.
\end{proof}
%\end{proof}

\begin{corollary} \label{connected Kodaira dimension 1}
Suppose $X \in Sm/k$ is an $\mathbb{A}^1$-connected $k$-variety, where $k$ is an algebraically closed field of characteristic zero. Then $X$ has negative logarithmic Kodaira dimension.
\end{corollary}
\begin{proof}
Since $X$ is $\mathbb{A}^1$-connected, by Theorem \ref{A1-connected dominant}, there
 %surface and $Sing_*(X)$ is Kan fibrant in degree $2$, by Proposition \ref{connected implies} in both cases there 
is a non-constant homotopy $H: \mathbb{A}^1_k \times_k Y \to X$ such that $Y$ is irreducible and $H$ is dominant. Infact we can take $dim(Y) = dim(X) -1$. By \cite[Theorem 11.3]{iitaka}, the variety $\mathbb{A}^1_k \times_k Y$ has logarithmic Kodaira dimension $-\infty$. Thus $X$ has logarithmic Kodaira dimension $-\infty$ by \cite[Proposition 11.4]{iitaka}. 
%We can assume $C$ to be non-singular. Since $H$ is non-constant homotopy, there is some $P \in \mathbb{A}^1_C(k)$ such that the differential at $P$ is surjective. 
%Indeed, suppose we have a non-constant homotopy $H:\mathbb{A}^1_C \to X$ which is dominant. There is some $t \in \mathbb{A}^1_k$ such tha $\overline{Im(H(t, -))}$ has dimension at least $1$. If there is $t^{\prime}$ such that $\overline{Im(H(t, -))} \neq $\overline{Im(H(t^{\prime}, -))}$, there are at least two curves through some point of $X$. Otherwise all the closures are same and $H(t, -) \neq H(t^{\prime}, _)$ as morphisms to $X$. In this case one of the closure at least $2$ and the image of $H$ contains $\mathbb{A}^1$. So in this case also there are two distinct curves through some point. Thus at that point the differential is surjective.
%Thus $H$ is separable morphism \cite[Section 5.5]{hump}. Any $\mathbb{A}^1$-connected surface is rational. Indeed, suppose $Y$ is a smooth proper compactification of $X$. Then $Y$ is $\mathbb{A}^1$-connected \cite[Proposition 4.2.7]{as}. Any $\mathbb{A}^1$-connected smooth proper surface is separably rationally connected \cite[Corollary 2.4.4]{am}. So $Y$ is a rational surface \cite[Exercise IV.3.3.5]{kollar}. By \cite[Theorem 11.3]{iitaka}, the variety $\mathbb{A}^1_k \times_k Y$ has logarithmic Kodaira dimension $-\infty$. Thus $X$ has logarithmic Kodaira dimension $-\infty$ by \cite[Proposition 11.4]{iitaka}.
\end{proof}

%%%%%%%% Section 3 %%%%%%%%%%%%%%%%%%%%

\

%Suppose, $\mathcal{F}$ is a sheaf and $W \in Sm/k$. 
%Recall from \cite[Definition 4.4]{c}, a homotopy $H \in \mathcal{F}(\mathbb{A}^1_W)$ is called a non-constant homotopy if $H(0) \neq H(1) \in \mathcal{F}(W)$, where $H(0)$ and $H(1)$ are induced by the $0$-section and the $1$-section $W \to \mathbb{A}^1_W$ respectively. 
By the phrase ``there is an $\mathbb{A}^1$ in $W$" for some $W \in Sm/k$, we mean that there is a non-constant morphism $\phi: \mathbb{A}^1_k \to W$. If the base field $k$ is algebraically closed, then an $\mathbb{A}^1$-connected variety $X$ has always a non-constant $\mathbb{A}^1$ in $X$ (Theorem \ref{A1-connected dominant}, \cite[Corollary 4.10]{cb}). But if $k$ is not algebraically closed, then it is not true (Remark \ref{not algebraically closed}). However, over a general field $k$ an $\mathbb{A}^1$-connected variety always admits a non-constant morphism $H: \mathbb{A}^1_L \to X$, for some finite separable field extension $L/k$ (Proposition \ref{separable non-constant}). 
\begin{lemma} \label{induction}
Let $k$ be a field and $X \in Sm/k$. Suppose that there is a non-constant homotopy $H \in \mathcal{S}^n(X)(\mathbb{A}^1_W)$ for some $W \in Sm/k$ such that $W$ is irreducible. Then there is a non-constant homotopy $H^{\prime} \in \mathcal{S}^{n-1}(X)(\mathbb{A}^1_{W^{\prime}})$ for some $W ^{\prime} \in Sm/k$ such that $W^{\prime}$ is irreducible.
\end{lemma}
\begin{proof}
The homotopy $H \in \mathcal{S}^n(X)(\mathbb{A}^1_W)$ is given by a Nisnevich covering $V \to \mathbb{A}^1_W$ along with $G \in \mathcal{S}^{n-1}(X)(V)$ such that there is a Nisnevich covering $V^{\prime} \to V \times_{\mathbb{A}^1_W} V$ with $p_1^*(G)|_{V^{\prime}} = p_2^*(G)|_{V^{\prime}}$ in $\mathcal{S}^{pre}(\mathcal{S}^{n-1}(X))(V^{\prime})$ (here $p_1, p_2 : V \times_{\mathbb{A}^1_W} V \to V$ are the projection maps). Therefore there are $G_1, G_2, \dots G_l \in \mathcal{S}^{n-1}(X)(\mathbb{A}^1_{V^{\prime}})$ such that $G_1(0) = p_1^*(G)|_{V^{\prime}}$ and $G_l(1) = p_2^*(G)|_{V^{\prime}}$. Since $H$ is non-constant, we can assume that $G_1$ is non-constant by Remark \ref{useful remark}. Thus there is some irreducible component $W^{\prime}$ of $V^{\prime}$ such that $G_1|_{\mathbb{A}^1_{W^{\prime}}}$ is non-constant.  
\end{proof}
\begin{proposition} \label{separable non-constant}
Let $k$ be any field (not necessarily algebraically closed) and $X \in Sm/k$ be an $\mathbb{A}^1$-connected scheme. Suppose $X$ has at least two $k$-rational points. Then there is a non-constant homotopy $H:\mathbb{A}^1_L \to X$ for some finite separable field extension $L$ of $k$.
\end{proposition}
\begin{proof}
Suppose, $X \in Sm/k$ is an $\mathbb{A}^1$-connected scheme and $\alpha, \beta$ are two $k$-rational points in $X$. If there is a non-constant $\mathbb{A}^1$ in $X$ through $\alpha$ or $\beta$, we are done. Let us assume that there is no non-constant $\mathbb{A}^1$ in $X$ through $\alpha$ and $\beta$. Since $X$ is $\mathbb{A}^1$-connected, $\mathcal{L}(X)$ is trivial \cite[Corollary 2.18]{bhs}. Therefore, there is an $n \geq 1$ such that $\alpha = \beta \in \mathcal{S}^{n+1}(X)(Spec \ k)$ but $\alpha \neq \beta \in \mathcal{S}^n(X)(Spec \ k)$. Thus there is a chain of non-constant $\mathbb{A}^1$-homotopies (Remark \ref{useful remark}) $H_1, H_2, \dots, H_m \in \mathcal{S}^n(X)(\mathbb{A}^1_k)$ such that $H_1(0) = \alpha$ and $H_m(1) = \beta$. Therefore applying Lemma \ref{induction} repeatedly, there is a non-constant homotopy $G :\mathbb{A}^1_W \to X$ for some $W \in Sm/k$ such that $W$ is irreducible. Since $G$ is the non-constant homotopy on $\mathbb{A}^1_W$ and the set
   $$\{w \in W \text{ closed point} \mid k \subset k(w) \text{ is a finite separable extension}\}$$
is dense in $W$ \cite[Tag 056U]{stack}, there is some $w_0 \in W$ such that $G(0)(w_0) \neq G(1)(w_0)$ as morphisms from $W$ to $X$ and $k(w_0)/k$ is a finite separable extension. Therefore the composition given by 
$$\mathbb{A}^1_{k(w_0)} \xrightarrow{w_0 \times Id}  \mathbb{A}^1_W \xrightarrow{G} X$$
is a non-constant morphism $H: \mathbb{A}^1_{k(w_0)} \to X$.
\end{proof}
\begin{remark} \label{not algebraically closed}
%In case of an algebraically closed field $k$, there is always a non-constant $\mathbb{A}^1$ in $X$ if $X$ is an $\mathbb{A}^1$-connected scheme \cite[Corollary 4.10]{c}. However, 
If $k$ is not an algebraically closed field, then the separable extension $L$ in Proposition \ref{separable non-constant} can be a proper extension of $k$. For example, if $X$ is the real sphere $Spec(\frac{\mathbb{R}[X, Y, Z]}{(X^2 + Y^2 + Z^2 - 1)})$, there is no non-constant morphism from $\mathbb{A}^1_{\mathbb{R}}$ to $X$. Indeed, if $\bar{X}$ is the projective closure of $X$, any morphism $\phi: \mathbb{A}^1_{\mathbb{R}} \to X$ extends to a morphism $\bar{\phi}: \mathbb{P}^1_{\mathbb{R}} \to \bar{X}$. Since $X$ has no real points at infinity, the morphism $\bar{\phi}$ factors through $X \hookrightarrow \bar{X}$. As $X$ is an affine scheme, the morphism $\bar{\phi}$ is constant. However the ring homomorphism $f: \frac{\mathbb{R}[X, Y, Z]}{(X^2+ Y^2+ Y^2 - 1)} \to \mathbb{C}[T]$ given by $X \mapsto 1,\ Y \mapsto T, \ Z \mapsto iT$ defines a non-constant morphism $\bar{f}: \mathbb{A}^1_{\mathbb{C}} \to X$.
\end{remark}

\

%The following lemma says that in case of a non-constant homotopy, the image of $\phi$ is not contained inside a single $\mathbb{A}^1$ in $X$.
\begin{lemma} \label{non-constant homotopy}
Suppose, $X$ is a smooth affine $k$-variety and $H \in \mathcal{S}^n(X)(\mathbb{A}^1_W)$ is a non constant homotopy, where $n \geq 1$ and $W$ smooth irreducible $k$-scheme. Let $f: V \to \mathbb{A}^1_W$ be a Nisnevich covering and suppose $\phi: V \to X$ is a morphism such that the epimorphism $\eta : X \to S^n(X)$ maps $\phi$ to $H|_V$. 

%Suppose that there is a Nisnevich covering $f: V \to \mathbb{A}^1_W$ and a morphism $\phi: V \to X$ such that the canonical epimorphism $X \to \mathcal{S}^n(X)$ maps $\phi$ to $H$. %
Then there are irreducible components $V_0$ and $V_0^{\prime}$ of $V$ such that there is no $\gamma: \mathbb{A}^1_k \to X$ so that the image $\gamma(\mathbb{A}^1_k)$ contains both $\overline{\phi(V_0)}$ and $\overline{\phi(V_0^{\prime})}$.
\end{lemma}
\begin{proof}
Consider the following commutative diagram:
 $$\begin{tikzcd}
 V \ar[r, "\phi"] \ar[d, "f"] &X \ar[d]\\
\mathbb{A}^1_{W} \ar[r,"H"] 
&\mathcal{S}^n(X)
\end{tikzcd}$$ 
 Suppose $V = \coprod_{i=1}^m V_i$, $V_i$-s are the irreducible components of $V$.

If possible, for every $i, j$ there is $\gamma_{i,j}: \mathbb{A}^1_k \to X$ such that its image (note that the image $\gamma_{i,j}(\mathbb{A}^1_k)$ is closed in $X$, since $X$ is affine and we denote $Im(\gamma_{i,j})$ by $L_{i, j}$) contains both $\overline{\phi(V_i)}$ and $\overline{\phi(V_j)}$. Here are two cases. 

\

\textbf{Case 1: }Suppose each $\overline{\phi(V_i)}$ is of dimension zero. Then $\overline{\phi(V_i)}$ is a singleton set for every $i$, since $V_i$ is irreducible. Suppose, $\overline{\phi(V_i)} = \{\alpha_i\} \subset X(k)$. Then $\alpha_i= \alpha_j \in \mathcal{S}^n(X)(Spec \ k)$ for every $i, j$. Indeed, $\alpha_i, \alpha_j \in L_{i, j}$ and $\gamma_{i,j}(t) = \gamma_{i,j}(0) \in \mathcal{S}^n(X)(Spec \ k)$, for every $t \in \mathbb{A}^1_k$. To prove this, consider the naive $\mathbb{A}^1$-homotopy 
$$\Tilde{\gamma_{i,j}}: \mathbb{A}^1_k \to X \text{ defined as } s \mapsto \gamma(st).$$
Thus $H|_V = \alpha_i \in \mathcal{S}^n(X)(V)$ for all $i$ (a $k$-point $\alpha$ is considered as an element of $\mathcal{S}^n(X)(V)$ as follows
$$V \to Spec \ k \xrightarrow{\alpha} X \to \mathcal{S}^n(X))$$ 
%same  as $$V \to Spec \ k \xrightarrow{\alpha_i} X \to \mathcal{S}^n(X)$$ 
and hence $H = \alpha_i \ \forall i$. Therefore $H(0) = H(1) \in \mathcal{S}^n(X)(W)$, which is a contradiction since $H$ is a non-constant homotopy. 

\

\textbf{Case 2: }%Then there is an $\mathbb{A}^1$ in $X$ that contains all $\overline{\phi(V_i)}$'s. If each $\overline{\phi(V_i)}$ is a sing
Suppose, there is some $t$ such that $dim(\overline{\phi(V_t)}) \geq 1$. Then $\overline{\phi(V_t)} \subset L_{t, i}$ for every $i$. Since, $L_{t,i}$ is the image of $\mathbb{A}^1$, so $L_{t,i} = \overline{\phi(V_t)}, \ \forall i$. Thus $\overline{\phi(V_t)}$, which is the image of an $\mathbb{A}^1$ in $X$, contains all the $\overline{\phi(V_i)}$'s.
%Thus $L_{t, i}$ contains all the $\overline{\phi(V_i)}$'s is of dimension
%(here irreducible components are connected components also). 
%Thus $V \times_{\mathbb{A}^1_W} V$ is the union of $V_i \times_{\mathbb{A}^1_W} V_j$ varying $i$ and $j$ and each irreducible component of $V^{\prime}$ which is also an connected component, maps to $V_i \times_{\mathbb{A}^1_W} V_j$ for some $i, j$.
So there is an $\mathbb{A}^1$, say $\gamma: \mathbb{A}^1_k \to X$ such that the image $\gamma(\mathbb{A}^1_k)$ contains $\overline{\phi(V_i)}$ for every $i$ (we denote the image $\gamma(\mathbb{A}^1_k)$ by $\mathbf{L}$, which is closed in $X$). We will prove that in this situation $H$ is not a non-constant homotopy i.e. $H(0) = H(1) \in \mathcal{S}^n(X)(W)$. The $0$-section $H(0)$ is given by the following commutative diagram:
 $$\begin{tikzcd}
U_0 \ar[r, "\theta_0"] \ar[d, "f_0" left] &V \ar[r, "\phi"] \ar[d, "f"] &X \ar[d, "\eta" right]\\
W \ar[r,"i_0" below] &\mathbb{A}^1_{W} \ar[r, "H" below] &\mathcal{S}^n(X)
\end{tikzcd}$$ 
Here $i_0$ is the $0$-section and the left most square is a pullback square. Similarly the $1$-section $H(1)$ is given by the following commutative diagram:
$$\begin{tikzcd}
U_1 \ar[r, "\theta_1"] \ar[d, "f_1" left] &V \ar[r, "\phi"] \ar[d, "f"] &X \ar[d, "\eta" right]\\
W \ar[r,"i_1" below] &\mathbb{A}^1_{W} \ar[r, "H" below] &\mathcal{S}^n(X)
\end{tikzcd}$$
Here $i_1$ is the $1$-section and the left most square is a pullback square. Here are two possible cases regarding the closures of the images of $\phi \circ \theta_0$ and $\phi \circ \theta_1$. \par
 \textbf{Subcase 1:} Assume that both the closures of $Im(\phi \circ \theta_0)$ and $Im(\phi \circ \theta_1)$ are $\mathbf{L}$. In this case, both $\phi \circ \theta_0$ and $\phi \circ \theta_1$ factor through the normalisation of $\mathbf{L}$ as $U_0 \xrightarrow{\Tilde{\theta_0}} \mathbb{A}^1_k \to X$ and $U_1 \to \mathbb{A}^1_k \to X$ respectively. Therefore we have the following commutative diagram for $\phi \circ \theta_0$:
$$\begin{tikzcd}
\mathcal{S}^n(U) \ar[r] &\mathcal{S}^n(\mathbb{A}^1_k) \ar[r, "\mathcal{S}^n(\gamma)"] &\mathcal{S}^n(X) \\
&\mathbb{A}^1_k \ar[dr, "\gamma"]  \ar[u, "\eta^{\prime}"]\\
U_0 \ar[r, "\theta_0"]  \ar[ur, "\tilde{\theta_0}"] \ar[uu] \ar[d, "f_0" left] &V \ar[r, "\phi"] \ar[d, "f"] &X \ar[d, "\eta" right] \ar[uu, "\eta" right] \\
W \ar[r,"i_0" below] &\mathbb{A}^1_{W} \ar[r, "H" below] &\mathcal{S}^n(X)
%U_0 \ar[r] \ar[d] &\mathbb{A}^1_k \ar[r] \ar[d] &X \ar[d]\\
%\mathcal{S}^n(U_0) \ar[r] &\mathcal{S}^n(\mathbb{A}^1_{k}) \ar[r] &\mathcal{S}^n(X)
\end{tikzcd}$$ 
From the above diagram we have,
\begin{align*}
H \circ i_0 \circ f_0 &= H \circ f \circ \theta_0 \\
                             & = \eta \circ \phi \circ \theta_0 \\
                            & = \eta \circ \gamma \circ \tilde{\theta_0} \\
                            & = \mathcal{S}^n(\gamma) \circ \eta^{\prime} \circ \tilde{\theta_0}
\end{align*}
Since $\mathcal{S}^n(\mathbb{A}^1_k)$ is the trivial sheaf $Spec \ k$, therefore, $H(0)|_{U_0}$ is same as the composition of maps $U_0 \to \Spec \ k \xrightarrow{\alpha} \mathcal{S}^n(X)$ for some $\alpha \in X(k)$. Now, $\gamma = \gamma(0) \in \mathcal{S}(X)(\mathbb{A}^1_k)$. Indeed, a naive $\mathbb{A}^1$-homotopy between $\gamma$ and $\gamma(0)$ is given by 
$$\tilde{\gamma}: \mathbb{A}^1_k \times_k \mathbb{A}^1_k \to X \text{ defined as } (s,t) \mapsto \gamma(st).$$ 
Thus $\alpha = \gamma(0) \in \mathcal{S}^n(X)(Spec \ k)$ and $H(0)|_{U_0} = \gamma(0)$.
Similarly, we have the same kind diagram for $\phi \circ \theta_1$ and $H(1)|_{U_1} = \gamma(0)$. 
%Since $\mathcal{S}(\mathbb{A}^1_k)$ is isomorphic to $\mathcal{S}(Spec \ k)$ i.e. $Spec \ k$ and both $H(0)|_{U_0}$ and $H(1)|_{U_1}$ factor through $\mathcal{S}^n(\mathbb{A}^1_k)$, therefore $H(0)$ and $H(1)$ are equal to some $\alpha \in \mathcal{S}^n(X)(Spec \ k)$. 
Thus both $H(0) = H(1) = \gamma(0)$ and this contradicts the fact that $H$ is a non-constant homotopy. \par
\textbf{Subcase 2:} Assume that one of the closures (say $\overline{Im(\phi \circ \theta_0)}$) consists finitely many points in $\mathbf{L}$. Suppose, $\overline{Im(\phi \circ \theta_0)} = \{\alpha_1, .., \alpha_m\} \subset X(k)$ where $\alpha_i = \gamma(t_i)$, for some $t_i \in \mathbb{A}^1_k$ . So each irreducible component of $U_0$ maps to some $\alpha_i$ under $\phi \circ \theta_0$. Thus for each irreducible component (say $U_0^{\prime}$) of $U_0$, $\phi \circ \theta_0|_{U_0^{\prime}}$ factors as $U_0^{\prime} \xrightarrow{t_i} \mathbb{A}^1_k \xrightarrow{\alpha_i} X$.
% via a pre-image in $\mathbb{A}^1_k$ of the constant image of $U^{\prime}$ in $\mathbf{L}$. 
Hence in this case also $H(0) = \gamma(0)$, since $\gamma = \gamma(t_0) \in \mathcal{S}^n(X)(\mathbb{A}^1_k)$, for every $t_0 \in \mathbb{A}^1_k$.  Indeed, a naive $\mathbb{A}^1$-homotopy between $\gamma$ and $\gamma(t_0)$ is given by 
$$\tilde{\gamma}: \mathbb{A}^1_k \times_k \mathbb{A}^1_k \to X \text{ defined as } (s,t) \mapsto \gamma((st+(1-s)t_0)).$$ 
 %factors as $U_0 \to \mathbb{A}^1_k \to X$. Thus as in Case 1, $H(0)|_{U_0}$ factors through $\mathcal{S}^n(\mathbb{A}^1_k)$.
Similarly, $H(1)= \gamma(0)$. 
%Therefore both $H(0)$ and $H(1)$ equal to some $\alpha \in \mathcal{S}^n(X)(Spec \ k)$. 
It is a contradiction to the fact that $H$ is non-constant. \\
Hence the Lemma follows. 
\end{proof}

%\
%%The following Proposition strengthens \cite[Theorem 4.9]{cb} in case of affine surface. By the term ``two $\mathbb{A}^1$'s (given by $\gamma_1, \gamma_2: \mathbb{A}^1_k \to X$) in $X$ intersect", we mean that $Im(\gamma_1) \cap Im(\gamma_2) \neq \emptyset$. 
From now on by ``two $\mathbb{A}^1$'s (given by $\gamma_1, \gamma_2: \mathbb{A}^1_k \to X$) in $X$ intersect", we mean that $Im(\gamma_1) \cap Im(\gamma_2) \neq \emptyset$.
\begin{proposition} \label{connected implies}
Let $X \in Sm/k$ be an affine surface such that $\pi_0^{\mathbb{A}^1}(X)(Spec \ k)$ is trivial. Then one of the following holds:
\begin{enumerate}
\item There exists some $Y \in Sm/k$ such that $Y$ is irreducible along with a non-constant homotopy $H:\mathbb{A}^1_k \times Y \to X$ which is dominant.
\item $\forall x \in X(k)$ there is an $\mathbb{A}^1$ in $X$ and there are $k$-points $\alpha, \ \beta$ in $X$ and two distinct $\mathbb{A}^1$-s (the images are distinct) in $X$ through $\alpha$ and $\beta$ respectively such that they intersect.
\end{enumerate}
\end{proposition}
\begin{proof} 
Suppose, $X \in Sm/k$ is an affine surface with $\pi_0^{\mathbb{A}^1}(X)(Spec \ k)$ is trivial. If there is some $\alpha \in X(k)$ such that there is no non-constant $\mathbb{A}^1$ through $\alpha$, according to the proof in \cite[Theorem 4.9]{cb}, there is a non-constant homotopy $H: \mathbb{A}^1_k \times Y \to X$ which is dominant for some $Y \in Sm/k$ such that $Y$ is irreducible. If possible, the conclusion of this Proposition is false for $X$. Thus we assume that for each $x \in X(k)$ there is the unique $\mathbb{A}^1$ through $x$ (this means given two $\mathbb{A}^1$'s $\gamma_1, \gamma_2$ through $x$, we have $Im(\gamma_1) = Im(\gamma_2)$) i.e. there are no intersecting $\mathbb{A}^1$'s in $X$ and $X$ does not admit such a dominant map as described in the conclusion of the Proposition. Fix $\alpha, \beta \in X(k)$ such that $\beta$ lies outside the unique $\mathbb{A}^1$ through $\alpha$. Since $X$ is $\mathbb{A}^1$-connected, $\mathcal{L}(X)$ is trivial \cite[Corollary 2.18]{bhs}. Therefore there is an $n \geq 1$ such that $\alpha = \beta \in \mathcal{S}^{n+1}(X)(Spec \ k)$, but $\alpha \neq \beta \in \mathcal{S}^n(X)(Spec \ k)$. Thus there is a chain of non-constant $\mathbb{A}^1$-homotopies (Remark \ref{useful remark}) $H_1, H_2, \dots, H_p \in \mathcal{S}^n(X)(\mathbb{A}^1_k)$ such that $H_1(0) = \alpha$ and $H_p(1) = \beta$. By applying Lemma \ref{induction} repeatedly, there is a non-constant homotopy $G \in \mathcal{S}(X)(\mathbb{A}^1_W)$ for some $W \in Sm/k$ such that $W$ is irreducible. Since the morphism $X \to \mathcal{S}(X)$ is an epimorphism, there is a Nisnevich covering $f: V \to \mathbb{A}^1_W$ and a morphism $\phi: V \to X$ such that the following diagram commutes:
 $$\begin{tikzcd}
 V \ar[r, "\phi"] \ar[d, "f"] &X \ar[d]\\
\mathbb{A}^1_{W} \ar[r,"G"] 
&\mathcal{S}(X)
\end{tikzcd}$$
%We have $G \in \mathcal{S}(X)(\mathbb{A}^1_W)$, $G$ is given by the morphism $\phi: V \to X$ and
There is a Nisnevich covering $V^{\prime} \to V \times_{\mathbb{A}^1_W} V$ such that $\phi \circ p_1|_{V^{\prime}}$ and $\phi \circ p_2|_{V^{\prime}}$ are $\mathbb{A}^1$-chain homotopic (where $p_1, p_2: V \times_{\mathbb{A}^1_W} V \to V$ are the projections). Thus there is a chain of non-constant $\mathbb{A}^1$-homotopies (Remark \ref{useful remark}) $G_1, G_2, \dots, G_m : \mathbb{A}^1_{V^{\prime}} \to X$ such that $G_1(0) = \phi \circ p_1|_{V^{\prime}}$ and $G_m(1) = \phi \circ p_2|_{V^{\prime}}$. \par
      Suppose $V = \coprod_i V_i$, $V_i$-s are the irreducible components of $V$. 
%Thus $V \times_{\mathbb{A}^1_W} V$ is the union of $V_i \times_{\mathbb{A}^1_W} V_j$ varying $i$ and $j$ and 
Then each irreducible component of $V^{\prime}$ (which is also a connected component) maps to $V_i \times_{\mathbb{A}^1_W} V_j$ for some $i, j$ (note that, each $V_i \times_{\mathbb{A}^1_W} V_j$ is non-empty since $W$ is irreducible). If an irreducible component (say $V_0$) of $V^{\prime}$ maps to $V_q \times_{\mathbb{A}^1_W} V_s$ such that $\overline{\phi(V_q)}$ and $\overline{\phi(V_s)}$ are distinct, there is some $t$ such that $G_t|_{\mathbb{A}^1_{V_0}}$ is a non-constant homotopy and $G_t|_{\mathbb{A}^1_{V_0}}(0) = \phi \circ p_1|_{V_0}$, since we have $G_1|_{\mathbb{A}^1_{V_0}}(0) = \phi \circ p_1|_{V_0}$ and $G_m|_{\mathbb{A}^1_{V_0}}(1) = \phi \circ p_2|_{V_0}$ and $\overline{Im(\phi \circ p_1|_{V_0})} = \overline{\phi(V_q)}$ and $\overline{Im(\phi \circ p_2|_{V_0})} = \overline{\phi(V_s)}$. As $k$ is algebraically closed, there is a non-constant $\mathbb{A}^1$ contained in $Im(G_t|_{\mathbb{A}^1_{V_0}})$. Since $G$ is a non-constant homotopy, there is some $i$ and $j$ such that there is no $\mathbb{A}^1$ in $X$ that contains both $\overline{\phi(V_i)}$ and $\overline{\phi(V_j)}$ by Lemma \ref{non-constant homotopy}. The rest of the proof follows from the following cases. \par

\

\begin{figure} [h!]
  \includegraphics[width=0.4\linewidth]{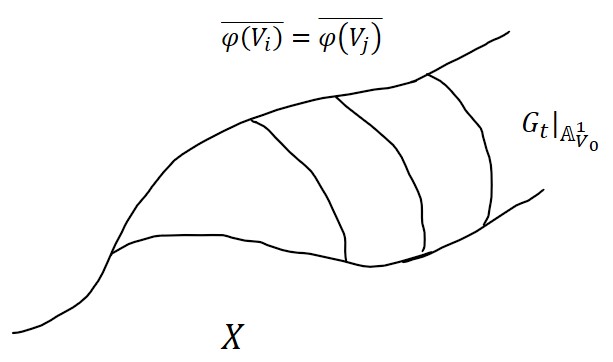}
  \caption*{Case 1}
  %\label{fig:boat1}
\end{figure}

\textbf{Case 1:} Assume that $\overline{\phi(V_i)}$ and $\overline{\phi(V_j)}$ are equal and there is some $r$ such that $\overline{\phi(V_r)}$ and $\overline{\phi(V_i)}$ are distinct. There is some irreducible component of $V^{\prime}$ (say $V_0$) that maps to $V_i \times_{\mathbb{A}^1_W} V_r$. There is some $t$ such that $G_t|_{\mathbb{A}^1_{V_0}}$ is non-constant and $G_t|_{\mathbb{A}^1_{V_0}}(0) = \phi \circ p_1 |_{V_0}$. Here $\overline{Im(\phi \circ p_1|_{V_0})}$ is same with $\overline{\phi(V_i)}$. Since $G_t|_{\mathbb{A}^1_{V_0}}$ is non-constant and $k$ is algebraically closed, $Im(G_t|_{\mathbb{A}^1_{V_0}})$ contains a non-constant $\mathbb{A}^1$ and $\overline{Im(G_t|_{\mathbb{A}^1_{V_0}})}$ contains $\overline{\phi(V_i)}$. Therefore $G_t|_{\mathbb{A}^1_{V_0}}$ is  dominant, since $\overline{\phi(V_i)}$ is not contained in a single $\mathbb{A}^1$. This is a contradiction since we have assumed that $X$ does not admit such a dominant map. \par

\

\begin{figure} [h!]
  \includegraphics[width=0.4\linewidth]{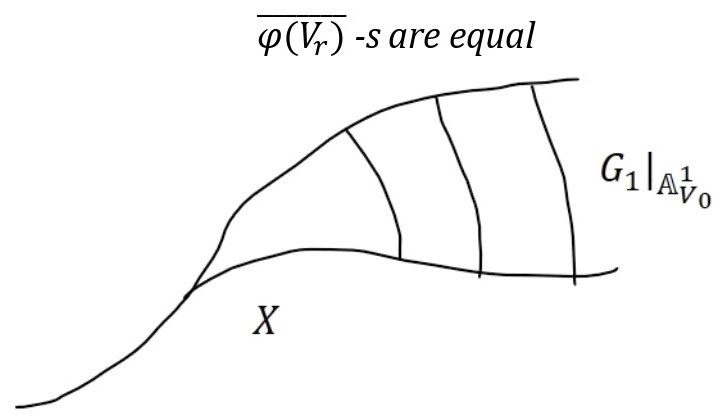}
  \caption*{Case 2}
  %\label{fig:boat1}
\end{figure}

\textbf{Case 2:} Assume that for every $r$, $\overline{\phi(V_r)}$ are same. Choose an irreducible component of $V^{\prime}$ (say $V_0$) such that $G_1|_{\mathbb{A}^1_{V_0}}$ is non-constant. Suppose $V_0$ maps to $V_l \times_{\mathbb{A}^1_W} V_s$ for some $l, s$. Here $G_1(0) = \phi \circ p_1 |_{V^{\prime}}$ and $\overline{Im(\phi \circ p_1|_{V_0})}$ is same with $\overline{\phi(V_l)}$. Since $G_1|_{\mathbb{A}^1_{V_0}}$ is non-constant and $k$ is algebraically closed, $Im(G_1|_{\mathbb{A}^1_{V_0}})$ contains a non-constant $\mathbb{A}^1$ and $\overline{Im(G_1|_{\mathbb{A}^1_{V_0}})}$ contains $\overline{\phi(V_l)}$. Therefore $G_1|_{\mathbb{A}^1_{V_0}}$ is  dominant, since $\overline{\phi(V_l)}$ is not contained in a single $\mathbb{A}^1$. It is a contradiction. \par

\

\textbf{Case 3: } Assume that $\overline{\phi(V_i)}$ and $\overline{\phi(V_j)}$ are distinct. There is an irreducible component (say $V_0$) of $V^{\prime}$ that maps to $V_i \times_{\mathbb{A}^1_W} V_j$. Then there is some $t$ such that $\overline{Im(G_t|_{\mathbb{A}^1_{V_0}}(0))}$ and $\overline{Im(G_t|_{\mathbb{A}^1_{V_0}}(1))}$ are distinct and there is no single $\mathbb{A}^1$ in $X$ that contains both $\overline{Im(G_t|_{\mathbb{A}^1_{V_0}}(0))}$ and $\overline{Im(G_t|_{\mathbb{A}^1_{V_0}}(1))}$. 
Indeed if possible, assume that for each $l$ either $\overline{Im(G_l|_{\mathbb{A}^1_{V_0}}(0))}$ and $\overline{Im(G_l|_{\mathbb{A}^1_{V_0}}(1))}$ are same or if $\overline{Im(G_l|_{\mathbb{A}^1_{V_0}}(0))}$ and $\overline{Im(G_l|_{\mathbb{A}^1_{V_0}}(1))}$ are distinct, then there is an $\mathbb{A}^1$ in $X$ that contains both the closures. Thus for each homotopy $G_l|_{\mathbb{A}^1_{V_0}}$, there is an $\mathbb{A}^1$ in $X$ that contains both $\overline{Im(G_l|_{\mathbb{A}^1_{V_0}}(0))}$ and $\overline{Im(G_l|_{\mathbb{A}^1_{V_0}}(1))}$, as $G_l(1) = G_{l+1}(0)$. Since there are no intersecting $\mathbb{A}^1$-s in $X$ and $G_l(1) = G_{l+1}(0)$, there is a single $\mathbb{A}^1$ in $X$ that contains all $\overline{Im(G_l|_{\mathbb{A}^1_{V_0}}(0))}$ and $\overline{Im(G_l|_{\mathbb{A}^1_{V_0}}(1))}$ for any $l$. Thus there is an $\mathbb{A}^1$ in $X$ that contains both $\overline{\phi(V_i)}$ and $\overline{\phi(V_j)}$. It is a contradiction.
%$\overline{\phi(V_i)}$ and $\overline{\phi(V_j)}$, since any two distinct $\mathbb{A}^1$-s in $X$ are disjoint by our assumption).
Since $k$ is algebraically closed, $Im(G_t|_{\mathbb{A}^1_{V_0}})$ contains a non-constant $\mathbb{A}^1$ and $\overline{Im(G_t|_{\mathbb{A}^1_{V_0}}(0))}$ and $\overline{Im(G_t|_{\mathbb{A}^1_{V_0}}(1))}$ are not contained in that $\mathbb{A}^1$. Therefore, $G_t|_{\mathbb{A}^1_{V_0}}$ is dominant. It is a contradiction. 

\begin{figure} [h!]
  \includegraphics[width=0.4\linewidth]{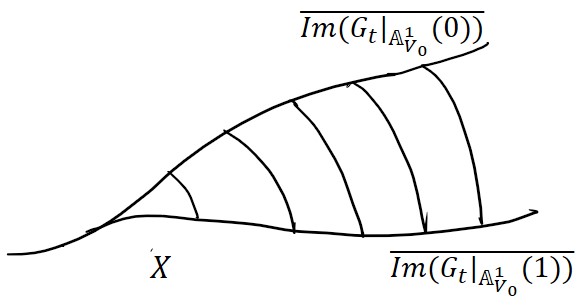}
  \caption*{Case 3}
  %\label{fig:boat1}
\end{figure}

Hence the Proposition follows.
\end{proof}

Recall that a simplicial set $X$ is called Kan fibrant \cite[Definition 1.3]{may} if for every $n, \ l$, $0 \leq l \leq n$ 
%and given a map $\alpha: \Lambda^n_l \to X$, there is a map $\bar{\alpha}: \Delta^n \to X$ such that $\alpha = \bar{\alpha} \circ \theta$, where $\theta: \Lambda^n_l \to \Delta^n$ is the inclusion of the $l$-th horn in $\Delta^n$. Equivalently,
given $n$-many $(n-1)$ simplices $x_0,.., x_{l-1}, x_{l+1},.., x_n$ of $X$ satisfying the compatibility condition $d_i x_j = d_{j-1} x_i$ for $i < j, i, j \neq l$, there is an $n$-simplex $x$ such that $d_i x  = x_i$ for all $i \neq l$.  
\begin{definition}
A simplicial set $X$ is called Kan fibrant in degree $n$ if for each $l$-th horn $\Lambda^n_l$ where $0 \leq l \leq n$, a map $\phi: \Lambda^n_l \to X$ can be extended to $\Tilde{\phi}: \Delta^n \to X$. 
\end{definition}
\begin{example}
%\begin{enumerate}
Any Kan fibrant simplicial set is Kan fibrant in degree $n$ for any $n$. The retract of a Kan fibrant simplicial set is Kan fibrant. Any simplicial group is Kan fibrant \cite[Theorem 17.1]{may}. So $Sing_*(\mathbb{A}^m_k)$ \cite[Section 2.3]{mv} is sectionwise Kan fibrant.
%Any simplicial group is Kan fibrant. For example, $Sing_*(\mathbb{A}^m)$ is sectionwise Kan fibrant. By the phrase "$Sing_*(X)$ is Kan fibrant" for a $k$-variety $X$, we mean $Sing_*(X)(Spec \ k)$ is Kan fibrant. If $Sing_*(X)$ and $Sing_*(Y)$ are Kan fibrants for $k$-varieties $X$ and $Y$, $Sing_*(X \times Y)$ is also Kan fibrant. 

Thus in particular if we are given two morphisms $f, g : \mathbb{A}^1_k \to \mathbb{A}^m_k$ such that $f(1) = g(0)$, then the morphism
$h: \mathbb{A}^2_k \to \mathbb{A}^m_k$ defined as
$$(t_1, t_2) \mapsto (f(1-t_1)+g(t_2) - f(1))$$
satisfies $h(1-x, 0) = f(x)$ and $h(0, x) = g(x)$. So $Sing_*(\mathbb{A}^m_k)(Spec \ k)$ is Kan fibrant in degree $2$. Therefore, if there are two intersecting $\mathbb{A}^1$'s in $\mathbb{A}^2_k$, then we can extend it to get a morphism from $\mathbb{A}^2_k$ to $\mathbb{A}^2_k$.
%\end{enumerate}
\end{example}

\begin{corollary} \label{connected Kodaira dimension 2}
Let $X \in Sm/k$ be an affine surface, where $k$ is an algebraically closed field of characteristic zero. Suppose that $\pi_0^{\mathbb{A}^1}(X)(Spec \ k)$ is trivial and $Sing_*(X)(Spec \ k)$ is Kan fibrant in degree $2$. Then $X$ has negative logarithmic Kodaira dimension.
\end{corollary}
\begin{proof}
Since $\pi_0^{\mathbb{A}^1}(X)(Spec \ k)$ is trivial, so by the Proposition \ref{connected implies} either there is some $Y \in Sm/k$ irreducible along with a non-constant homotopy $H: \mathbb{A}^1_k \times_k Y \to X$ which is dominant or there are two intersecting $\mathbb{A}^1$'s in $X$. For the first case, we proceed as in Corollary \ref{connected Kodaira dimension 1} to conclude that $X$ has negative logarithmic Kodaira dimension. For the second case, since $Sing_*(X)(Spec \ k)$ is Kan fibrant in degree $2$, there is a dominant morphism $\phi: \mathbb{A}^2_k \to X$. Therefore $X$ has logarithmic Kodaira dimension $-\infty$ by \cite[Proposition 11.4]{iitaka}.
\end{proof}

\

\subsection{Characterisation of Affine Surface over a DVR} %%%%%%%%%%%%%%%%%% Section 4 %%%%%%%%%%%%%%%%%%%%
In this subsection we prove that over a discrete valuation ring $R$ of equicharacteristic zero, $\mathbb{A}^1$-contractibilty detects $\mathbb{A}^2_R$ (Theorem \ref{dvr}). However, if $R$ is not of equicharacteristic zero, then it is not true (Remark \ref{nequi}).

\begin{theorem}\label{dvr}
Let $R$ be an equicharacteristic zero discrete valuation ring and $X$ be a smooth affine scheme over $R$ of relative dimension $2$. Then $X$ is $\mathbb{A}^1$-contractible if and only if $X$ is isomorphic to $\mathbb{A}^2_R$.
\end{theorem}

\begin{proof}
Let $K$ and $k$ be the fraction field and residue field of $R$ respectively. The base changes $X_K$ and $X_k$ of $X$ over $K$ and $k$ respectively are $\mathbb{A}^1$-contractible \cite[Corollary 1.24]{mv}. 
%by Lemma \ref{base change}. 
Thus $X_K$ and $X_k$ are isomorphic to $\mathbb{A}^2_K$ and $\mathbb{A}^2_k$ respectively by \cite[Theorem 1.1]{cb}. Therefore $X$ is isomorphic to $\mathbb{A}^2_R$ by \cite[Theorem 1]{sathaye}.
\end{proof}

\begin{remark} \label{nequi}
Theorem \ref{dvr} is not true if $R$ is not of equicharacteristic zero. The counterexample of such $X$ appeared in \cite[Theorem 5.1]{asa}. The affine scheme $X$ is $\mathbb{A}^1$-contractible since it is a retract of $\mathbb{A}^3_R$ and the base extensions of $X$ are isomorphic to $\mathbb{A}^2_K$ and $\mathbb{A}^2_k$ over the fraction field $K$ and residue field $k$ respectively. However, $X$ is not isomorphic to $\mathbb{A}^2_R$.
\end{remark}

The characterisation of affine surfaces using $\mathbb{A}^1$-homotopy theory yields the generalised Zariski cancellation:

\begin{corollary}
Suppose $X$ is a smooth affine scheme over $R$ of relative dimension $2$ and $Y$ is a smooth scheme over $R$, where $R$ is a discrete valuation ring of equicharacteristic zero. Suppose $X \times_R Y \cong \mathbb{A}^N_R$. Then $X \cong \mathbb{A}^2_R$.
\end{corollary}

\begin{proof}
%Suppose $X$ is a smooth affine surface over a perfect field $R=k$ and $Y$ is a smooth scheme over $k$. Since $X \times Y \cong \mathbb{A}^N_k$, $X$ is $\mathbb{A}^1$-contractible as $X$ is a retract of $\mathbb{A}^N_k$. Taking base extension over $\bar{k}$, the algebraic closure of $k$ we have, $X_{\bar{k}} \times_{\bar{k}} Y_{\bar{k}} \cong \mathbb{A}^N_{\bar{k}}$. Since $Sing_*(\mathbb{A}^N_{\bar{k}})$ is sectionwise Kan fibrant and $X_{\bar{k}}$ is a retract of $\mathbb{A}^N_{\bar{k}}$, $Sing_*(X_{\bar{k}})$ is sectionwise Kan fibrant. Therefore, $X \cong \mathbb{A}^2_k$ by Theorem \ref{perfect}. \par
    If $X \times_R Y \cong \mathbb{A}^N_R$, then $X$ is a retract of $\mathbb{A}^N_R$. Thus $X$ is an $\mathbb{A}^1$-contractible smooth affine scheme over $R$ of relative dimension $2$. Thus $X \cong \mathbb{A}^2_R$ by Theorem \ref{dvr}.
\end{proof}

\

\subsection{Complex Sphere in $\mathbb{A}^3_{\mathbb{C}}$}
%\section{complex sphere in $\mathbb{A}^3_{\mathbb{C}}$
We have seen that $\mathbb{A}^1$-connectivity of a smooth complex variety $X$ over $\mathbb{C}$ implies $X$ has logarithmic Kodaira dimension $-\infty$. For affine surfaces $X$ this implies $X$ contains a cylinder. But $\mathbb{A}^1$-connectivity of a smooth complex surface $X$ does not necessarily imply that $X(\mathbb{C})$ is simply connected at infinity. 
Consider the complex sphere, 
$$X = Spec \ \frac{\mathbb{C}[x,y,z]}{(x^2+y^2+z^2-1)}.$$
The complex sphere $X$ is an $\mathbb{A}^1$-connected surface with non-trivial Picard group. Thus in particular, $X$ is not isomorphic to $\mathbb{A}^2_{\mathbb{C}}$. \\
By change of variables, $X$ is isomorphic to $Spec \ \frac{\mathbb{C}[x, y, z]}{(xy-z(2-z))}$. $X$ is a smooth complex affine surface with non-trivial Picard group.
\begin{lemma}
$\mathcal{O}(X)$ is not a U.F.D.
\end{lemma}
\begin{proof}
The following product has two ways of factorization in $\mathcal{O}(X)$,
$$\bar{x}\bar{y} = \bar{z}\overline{2-z}$$
The factorization is not unique. Therefore, $\mathcal{O}(X)$ is not a U.F.D.
\end{proof}
\begin{remark}
So the Picard group of $X$ is non-trivial and hence $\mathbb{A}^1$-fundamental group of $X$ is non-trivial. 
\end{remark}
Consider the morphisms 
$$\phi, \psi: \mathbb{G}_m \times_{\mathbb{C}} \mathbb{A}^1_{\mathbb{C}} \to Spec \ \frac{\mathbb{C}[x, y, z]}{(xy-z(2-z))}$$
given by $\phi(s, t) = (s, \frac{t(2-t)}{s}, t)$ and $\psi(s, t) = (\frac{t(2-t)}{s}, s, t)$. Then $ \mathbb{G}_m \times_{\mathbb{C}} \mathbb{A}^1_{\mathbb{C}}$ is isomorphic to the open subsets $D(x)$ and $D(y)$ of $X$ by the morphisms $\phi$ and $\psi$ respectively, where
\[
  D(x) = \left\lbrace  P \in  Spec \ \frac{\mathbb{C}[x, y, z]}{(xy-z(2-z))} \;\middle|\;
   \bar{x} \notin P
  \right\rbrace
\]
and
\[
  D(y) = \left\lbrace  P \in  Spec \ \frac{\mathbb{C}[x, y, z]}{(xy-z(2-z))} \;\middle|\;
   \bar{y} \notin P
  \right\rbrace
\]
Thus $X$ contains cylinders. Hence $X$ has negative logarithmic Kodaira dimension. The proof in \cite[Theorem 4.3.4]{sawant} shows that the complex sphere $X$ is $\mathbb{A}^1$-chain connected, in particular $X$ is $\mathbb{A}^1$-connected. Note that the closure $\overline{X}$ of $X$ in $\mathbb{P}^3_{\mathbb{C}}$ is isomorphic to $\mathbb{P}^1_{\mathbb{C}} \times \mathbb{P}^1_{\mathbb{C}}$ and $\overline{X} \setminus X = \mathbb{P}^1_{\mathbb{C}}$ of degree $2$. Therefore, the fundamental group at infinity of $X$ is non trivial.

\section{$\mathbb{A}^1$-homotopy type of $S^{3,2}$ and $S^{5,3}$}
Let $k$ be a field of characteristic $0$.  In this section we prove that if an open subscheme of a smooth affine $k$-surface has same $\mathbb{A}^1$-homotopy type as $\mathbb{A}^2_{k} \setminus \{(0,0)\}$, then it is isomorphic to $\mathbb{A}^2_{k} \setminus \{(0,0)\}$ (Theorem \ref{poincare}). However, in dimension three we prove that the Koras Russell threefold of the first kind minus a point is $\mathbb{A}^1$-weakly equivalent to $\mathbb{A}^3_{\mathbb{C}} \setminus \{(0,0,0)\}$, but it is not isomorphic to $\mathbb{A}^3_{\mathbb{C}} \setminus \{(0,0,0)\}$ (Theorem \ref{not poincare}).

\begin{theorem} \label{poincare}
Let $k$ be a field of characteristic $0$  and $X$ be a smooth affine $k$-surface such that $U \subset X$ is a non empty open subscheme. Suppose that $U$ is isomorphic to $\mathbb{A}^2_k \setminus \{(0, 0)\}$ in $\mathbf{H}(k)$. Then $U$ and $\mathbb{A}^2_k \setminus \{(0,0)\}$ are isomorphic as $k$-varieties.
%Suppose, $X \in Sm/k$ is an affine scheme and $D$ is a locally principal closed subscheme of $X$. Then $X \setminus D$ is an  
\end{theorem}
\begin{proof}
Let $K/k$ be a field extension such that $K$ is uncountable and algebraically closed. 

\textbf{Step 1 :} 
The quasi-affine variety $U$ is not affine.
 Indeed, if $U$ is affine then $U_K:=U \times_k Spec \ K$ is affine and $U_K \cong 
\mathbb{A}^2_K \setminus \{(0,0)\}$ in $\mathbf{H}(K)$. This implies that $U_K$ is an affine variety with trivial Picard group and trivial group of units. Moreover, $\mathbb{A}^1$-connectedness of $\mathbb{A}^2_K \setminus \{(0,0)\}$ implies that $U_K$ has logarithmic Kodaira dimension $-  \infty$ (Corollary \ref{connected Kodaira dimension 1}). Therefore, $U_K \cong \mathbb{A}^2_K$ as $K$-varieties by \cite[Section 4.1]{miyanishi}. As $\mathbb{A}^2_K \setminus \{(0,0)\}$ is not $\mathbb{A}^1$-simply connected \cite[Theorem 6.40]{mor},  this is absurd.

\

\textbf{Step 2 :}
 By Noetherian property, we can embed $U$ in a smooth affine $k$-surface $\Tilde{X}$, which is also smallest in the sense that there is no smooth affine surface in between $U$ and $\Tilde{X}$ contained in $X$. The closed subscheme $\Tilde{X} \setminus U$ is finitely many closed points. Indeed, if there is an irreducible closed subset $D$ of codimension $1$ contained in $\Tilde{X} \setminus U$, then $D$ is an effective Cartier divisor which is a locally principal closed subscheme. Thus $\Tilde{X} \setminus D$ is affine which contradicts that $\Tilde{X}$ is the smallest. Threfore $U = \Tilde{X} \setminus \{p_1, \dots p_n \}$ where $p_i$'s are the closed points of $\Tilde{X}$.

\

\textbf{Step 3 :} The smooth $K$-scheme $U_K$ is isomorphic to $\mathbb{A}^2_K \setminus \{(0,0) \}$ in $\mathbf{H}(k)$. Therefore $U_K$ is $\mathbb{A}^1$-connected which implies that $U_K$ is a connected open subset of $\Tilde{X}_K := \Tilde{X} \times_k Spec \ K$. Thus,  $U_K = \Tilde{X}_K \setminus \{p_1^K, \dots, p_n^K \}$, where $p_i^K$ is the extension of the closed point $p_i$ to $K$, and therefore each $p_i^K$ is a finite disjoint union of finitely many $K$ points. As the connected components of the smooth scheme $\Tilde{X}_K$ has dimension $2$ and as $\Tilde{X}_K \setminus \{p_1^K, \dots, p_n^K \}$ is connected, therefore $\Tilde{X}_K$ is a connected smooth scheme. 

\

\textbf{Step 4 :}
Since $\Tilde{X}_K \setminus U_K$ is of pure codimension $2$, $\Tilde{X}_K$ has trivial Picard group by \cite[Proposition 6.5]{hart}. By Theorem \ref{A1-connected dominant}, since $U_K$ is $\mathbb{A}^1$-connected, there is a dominant morphism $H:\mathbb{A}^1_{K} \times_{K} W \to U_K$ such that $H(0, -) \neq H(1, -)$ with $W$ a smooth $K$-variety. Composing it with the inclusion $U_K \hookrightarrow \Tilde{X}_K$, we get a dominant morphism $H: \mathbb{A}^1_{K} \times_{K} W \to \Tilde{X}_K$. Therefore $\Tilde{X}_K$ has logarithmic Kodaira dimension $- \infty$. The restriction map $\mathcal{O}(\Tilde{X}_K) \to \mathcal{O}(U_K)$ is an isomorphism since $\Tilde{X}_K \setminus U_K$ is of pure codimension $2$. So $\Tilde{X}_K$ has trivial group of units. Therefore, $\Tilde{X}_K$ is isomorphic to $\mathbb{A}^2_K$ as $K$-varieties \cite[Section 4.1]{miyanishi}. As there is no non-trivial $\mathbb{A}^2$-form over the field characteristic $0$ \cite[Theorem 3]{kamb}, we get $\Tilde{X} \cong \mathbb{A}^2_k$ and hence $U \cong \mathbb{A}^2_k \setminus \{ p_1, \dots p_n\}$, as $k$-varieties.

\

\textbf{Step 5 :}
Next consider the Gysin trinagle (\cite[Theorem 15.15]{MWV}) in $\mathbf{DM}_{gm}(k, \mathbb{Z})$ 
$$M(U) \to M(\Tilde{X}) \to \oplus_{i =1}^n M(\kappa(p_i))(2)[4] \to M(U)[1].$$
Note that each of $M(\kappa(p_i))$ are strongly dualizable with $M(\kappa(p_i))^* = M(\kappa(p_i))$ (for a object $M \in \mathbf{DM}_{gm}(k, \mathbb{Z})$ $M^*$ is the dual  \cite{MWV}[Definition 20.6, Example 20.11, Definition 20.15]). Therefore, $M(\Tilde{X}) \cong \mathbb{Z} = M(Spec \ k)$ implies that the map $M(\Tilde{X}) \to  \oplus_{i =1}^n M(\kappa(p_i))(2)[4]$ is the zero map and $M(U) \cong M(k)  \oplus_{i =1}^n M(\kappa(p_i))(2)[3]$. But $M(U) \cong M(\mathbb{A}^2 \setminus \{(0,0) \} \cong M(k) \oplus M(k)(2)[3]$. This shows that $n =1$ and $p_1$ is a $k$-rational point (one can use $H^{4, 3}_{\mathcal{M}}(U, \mathbb{Z})$ and the fact $H^{4,3}_{\mathcal{M}}(k, \mathbb{Z}) = 0$ and $H^{1,1}_{\mathcal{M}}(X, \mathbb{Z}) =  \mathcal{O}(X)^*$). This completes the proof.
\end{proof}

\

Theorem \ref{poincare} is not true in case of quasi-affine complex threefold.
Suppose that $X$ is the Koras-Russell threefold of the first kind \cite{df} which is given by the equation 
$$x^mz = y^r + t^s + x \text{ in } \mathbb{A}^4_k,$$ 
where $k$ is an algebraically closed field of characteristic $0$, $m \geq 2$ and $r, s \geq 2$ are coprime integers. Consider the morphism,
$$\phi: X \to \mathbb{A}^3_{k} \text{ given by } (x, y, z, t) \mapsto (x, y, t).$$
Suppose, $p =(1,0,1,0)$ is a point in $X$. Then $\phi(p) = q = (1, 0, 0)$ and $\phi^{-1}(q) = p$. This gives the restriction morphism 
$$\bar{\phi}:X \setminus \{p\} \to \mathbb{A}^3_{k} \setminus \{q\}.$$ 
We show here that $X \setminus \{p\} \cong \mathbb{A}^3_{k} \setminus \{q\}$ in $\mathbf{H}(k)$, but $X \setminus \{p\}$ is not isomorphic to $\mathbb{A}^3_{k} \setminus \{q\}$ as $k$-varieties.  
\begin{lemma} \label{isomorphism normal bundle}
The induced map $d\phi_p: T_pX \to T_q\mathbb{A}^3_{k}$ is an isomorphism between the tangent spaces.
\end{lemma}
\begin{proof}
Suppose, $f(x, y, z, t) = x^mz - y^r - t^s - x$. Then $\nabla f (x, y, z, t)= (mx^{m-1}z - 1, -ry^{r-1}, x^m, -st^{s-1})$, so $\nabla f(1, 0, 1 ,0) = (m-1,0,1, 0)$. Thus the tangent space $T_pX$ of $X$ at $p$ is given by 
$$T_pX = \{(a,b,-(m-1)a,d)|a, b, d \in k\}$$
The map $d\phi_p: T_pX \to T_q\mathbb{A}^3_{k}$ is given by $(a,b,-(m-1)a,d) \mapsto (a, b, d)$. Therefore $d\phi_p$ is an isomorphism between the tangent spaces.
\end{proof}
The proof of the following lemma is same as in \cite[Example 2.21]{dpo}. We include it here for the sake of completeness.
\begin{lemma} \cite[Example 2.21]{dpo} \label{punctured russell threefold}
The quasi-affine threefold $X^{\prime}=X \setminus \{p\}$ is $\mathbb{A}^1$-chain connected.
\end{lemma}
\begin{proof}
Suppose, $F/k$ is a finitely generated field extension. Consider the projection map
%$p_x: X^{\prime} \to \mathbb{A}^1_{\mathbb{C}}$ given by $(x, y, z, t) \mapsto x$. The fiber over $\mathbb{G}_m$ is $(\mathbb{G}_m \times_{\mathbb{C}} \mathbb{A}^2_{\mathbb{C}}) \setminus \{\alpha\}$ (where $\alpha$ is a $\mathbb{C}$-point). It is $\mathbb{A}^1$-chain connected. The fiber over $0$ is $\mathbb{A}^1_{\mathbb{C}} \times_{\mathbb{C}} \Gamma_{2,3}$ (where $\Gamma_{2, 3}$ is the curve defined as $z^2 + t^3 = 0$). It is $\mathbb{A}^1$-chain connected. The map $\theta: \mathbb{A}^1_{\mathbb{C}} \to X^{\prime}$ given by $\lambda \mapsto (-\lambda, 2-\lambda, 1-\lambda, \lambda - 1)$ connects the points $(0, 2, 1, -1)$ and $(-1, 1, 0, 0)$ (Note that the point $p = (1, -1, 0, 0)$ does not lie in the image of $\theta$).
$p_x: X^{\prime} \to \mathbb{A}^1_{k}$ given by $(x, y, z, t) \mapsto x$. The fiber over a point $\alpha \in \mathbb{G}_m$ is $\mathbb{A}^2_k$, if $\alpha \neq 1$ and $\mathbb{A}^2_k \setminus \{(0,0)\}$, if $\alpha = 1$. 
%$(\mathbb{G}_m \times_{F} \mathbb{A}^2_{F}) \setminus \{\alpha\}$ (where $\alpha$ is a $k$-point). 
Thus the fiber over every point of $\mathbb{G}_m$ is $\mathbb{A}^1$-chain connected, in particular for every $\alpha \in \mathbb{G}_m$, any two $F$-points in $p_x^{-1}(\alpha)$ can be joined by a chain of $\mathbb{A}^1_F$'s. The fiber over $0$ is $\mathbb{A}^1_{k} \times_{k} \Gamma_{r,s}$ (where, $\Gamma_{r,s}$ is the curve in $\mathbb{A}^2_k$ defined as $y^r + t^s = 0$). Here also any two $F$-points can be joined by $\mathbb{A}^1_F$'s. Indeed, for an $F$-point $(t_1, t_2, t_3)$ in $\mathbb{A}^1_{F} \times_{F} \Gamma_{r,s}$, the naive $\mathbb{A}^1$-homotopy given by
$$\gamma: \mathbb{A}^1_F \to \mathbb{A}^1_{k} \times_{k} \Gamma_{r,s} \text{ as } v \mapsto (t_1v, t_2v^s, t_3v^r)$$
joins $(0,0,0)$ with $(t_1, t_2, t_3)$. To get the naive-$\mathbb{A}^1$-homotopy between the points in different fibers, we find the polynomials $y(v), t(v) \in k[v]$ such that $v^m$ divides $y(v)^r+t(v)^s+v$. Indeed if $r$ is even and $s$ is odd (similarly for $r$ is odd and $s$ is even), suppose that $y(v)$ and $t(v)$ are given by
$$y(v) = 1+a_0v+a_1v^2+\dots+a_{m-2}v^{m-1} \text{ and } t(v) = -1-v\dots -v^{m-1},$$
for some $a_i \in k$. We choose $a_i$ according to the co-effecients of $v^i$ is zero in $y(v)^r+t(v)^s+v$ for every $i \leq m-1$. 
%For the point $(1,1,3,1)$, consider the morphism 
The naive $\mathbb{A}^1$-homotopy $\theta: \mathbb{A}^1_F \to X^{\prime}$ given by
$$v \mapsto (\alpha v, \frac{y(\alpha v)^r+t(\alpha v)^s+\alpha v}{(\alpha v)^m}, y(\alpha v), t(\alpha v))$$
connects a point in $p_x^{-1}(0)$ with $p_x^{-1}(\alpha)$, $\alpha \in \mathbb{G}_m(k)$. 
%with the point $(1,1,3,1)$. 
Note that, the point $p  = (1,0,1,0)$ does not lie in the image of $\theta$. Indeed, if for some $v$, $\theta(v) = p$, then $\alpha v=1, y(\alpha v)=1, t(\alpha v) = 0$. But then $\frac{y(\alpha v)^r+t(\alpha v)^s+\alpha v}{(\alpha v)^m} = 2$.
%$$\theta: \mathbb{A}^1_{F} \to X^{\prime} \text{ given by } \lambda \mapsto (-\lambda, 2-\lambda, 1-\lambda, \lambda - 1)$$ 
%joins the points $(0, 2, 1, -1)$ and $(-1, 1, 0, 0)$ lying in different fibers (Note that the point $p = (1, -1, 0, 0)$ does not lie in the image of $\theta$).      
Therefore, $X^{\prime}$ is $\mathbb{A}^1$-chain connected.
\end{proof}
\begin{theorem} \label{not poincare}
$X \setminus \{p\}$ is $\mathbb{A}^1$-weakly equivalent to $\mathbb{A}^3_{k} \setminus \{q\}$, however they are not isomophic as  $k$-varieties.
\end{theorem}
\begin{proof}
First we prove that $X \setminus \{p\}$ is $\mathbb{A}^1$-weakly equivalent to $\mathbb{A}^3_{k} \setminus \{q\}$.
Consider the commutative diagram with rows are cofibre sequences:
$$\begin{tikzcd}[column sep=50pt, row sep=50 pt]
X \setminus \{p\} \ar[r] \ar[d, "\bar{\phi}"] &X \ar[r] \ar[d, "\phi" near start] &X/(X \setminus \{p\}) \ar[d]\\
\mathbb{A}^3_{k} \setminus \{q\} \ar[r]  &\mathbb{A}^3_{k} \ar[r] &\mathbb{A}^3_{k}/(\mathbb{A}^3_{k} \setminus \{q\})
\end{tikzcd}$$
The middle vertical map is an $\mathbb{A}^1$-weak equivalence, since $X$ is $\mathbb{A}^1$-contractible \cite[Theorem 1.1]{df}. By homotopy purity, $X/(X \setminus \{p\})$ is isomorphic to the Thom space of the normal bundle over $p$ \cite[\S 3.2, Theorem 2.23]{mv}. Since $p$ is a $k$-point of the smooth threefold $X$, the normal bundle over $p$ in $X$ is the trivial bundle of rank three over $p$. The right vertical map 
$$(\mathbb{P}^1_{k})^{\wedge 3} \cong X/(X \setminus \{p\}) \to \mathbb{A}^3_{k}/(\mathbb{A}^3_{k} \setminus \{q\}) \cong (\mathbb{P}^1_{k})^{\wedge 3},$$
is induced by $d\phi_p$.
Thus the right vertical map is also an $\mathbb{A}^1$-weak equivalence by Lemma \ref{isomorphism normal bundle} \cite[Lemma 2.1]{voev2}. Therefore, taking simplicial suspension, the map $\sum_s \bar{\phi}: \sum_s(X \setminus \{p\}) \to \sum_s(\mathbb{A}^3_{k} \setminus \{q\})$ is an $\mathbb{A}^1$-weak equivalence. Now $X \setminus \{p\}$ is $\mathbb{A}^1$-connected by Lemma \ref{punctured russell threefold} and $\pi_1^{\mathbb{A}^1}(X \setminus \{p\})$ is also trivial \cite[Theoem 4.1]{asokdoran}. Thus $\bar{\phi}$ is a $\mathbb{A}^1$-homology equivalence. Therefore $\bar{\phi}$ is an $\mathbb{A}^1$-weak equivalence \cite[Theorem 1.1]{yuri}. \par
%\end{proof}
%\begin{theorem}
 Now we show that $X \setminus \{p\}$ is not isomorphic to $\mathbb{A}^3_{k} \setminus \{q\}$ as $k$-varieties.
%\end{theorem}
%\begin{proof}
 %Suppose, if possible there is an isomorphism $\phi:\mathbb{A}^3_{\mathbb{C}} \setminus \{q\} \to X \setminus \{p\}$. Since $q$ is a codimension $3$ point, $\phi$ can be extended to a morphism $\bar{\phi}: \mathbb{A}^3_{\mathbb{C}} \to X$. Now we claim that $\bar{\phi}(q)= p$. Assume if possible $\bar{\phi}(q) \in X \setminus \{p\}$. Choose a morphism $\theta: \mathbb{G}_m \to \mathbb{A}^3_{\mathbb{C}} \setminus \{q\}$ such that the image of $\theta$ is closed in $\mathbb{A}^3_{\mathbb{C}} \setminus \{q\}$ (for example, $\theta(\lambda) = (1, \lambda, 0)$). Then $\theta$ can be extended to a morphism $\bar{\theta}: \mathbb{A}^1_{\mathbb{C}} \to \mathbb{A}^3_{\mathbb{C}}$. Since $\bar{\phi}(q) \in X \setminus \{p\}$, so $\phi \circ \bar{\theta}: \mathbb{A}^1_{\mathbb{C}} \to X\setminus \{p\}$ is a morphism. Composing it with the inverse of $\phi$, we have $\theta$ can be extended to a morphism $\mathbb{A}^1_{\mathbb{C}}$ to $\mathbb{A}^1_{\mathbb{C}} \setminus \{q\}$ which is a contradiction. Therefore $\bar{\phi}(q) = p$. So $\bar{\phi}$ is a bijective morphism.
%\end{proof}
%\begin{proof}
Suppose, if possible there is an isomorphism $\phi:\mathbb{A}^3_{k} \setminus \{q\} \to X \setminus \{p\}$ with its inverse $\psi$. Since $q$ and $p$ are the codimension $3$ points of $\mathbb{A}^3_{k}$ and $X$ respectively, both $\phi$ and $\psi$ can be extended to a morphism $\bar{\phi}: \mathbb{A}^3_{k} \to X$ and $\bar{\psi}: X \to \mathbb{A}^3_{k}$. Both the maps $\bar{\psi} \circ \bar{\phi}$ and $\bar{\phi} \circ \bar{\psi}$ agree with the identity maps in a complement of a $k$-point. Therefore both $\bar{\phi}$ and $\bar{\psi}$ are  isomorphisms. It is a contradiction since $X$ has non trivial Makar-Limanov invariant \cite{kaliman}. Therefore, $X \setminus \{p\}$ is not isomorphic to $\mathbb{A}^3_k \setminus \{q\}$ as $k$-varieties.
\end{proof}

\end{document}